\documentclass[10pt,leqno,twoside,a4paper]{article}
\usepackage{amsfonts,amsmath,amsthm,amssymb}
\usepackage[english]{babel}
\newtheorem{theorem}{Theorem}
\newtheorem{proposition}[theorem]{Proposition}
\newtheorem{lemma}[theorem]{Lemma}
\newtheorem{corollary}[theorem]{Corollary}
\theoremstyle{definition}

\theoremstyle{remark}

 \numberwithin{equation}{section}

\newcommand{\si}{\sigma}

\begin{document}
\title{
 A description of quasi-duo $\mathbb{Z}$-graded rings}
 \author{ {
   Andr\'{e} Leroy$^\dagger$,   Jerzy
Matczuk\footnote{ The research was
 supported by Polish MNiSW grant No. N N201 268435},  Edmund R. Puczy{\l}owski$^* $}\\
\\   $^\dagger$ \normalsize Universit\'{e} d'Artois,  Facult\'{e} Jean Perrin\\
\normalsize Rue Jean Souvraz  62 307 Lens, France\\
   \normalsize  e.mail: leroy@poincare.univ-artois.fr\\
 \\ $^*$ \normalsize Institute of Mathematics, Warsaw University,\\
 \normalsize Banacha 2, 02-097 Warsaw, Poland\\
 \normalsize e.mail: jmatczuk@mimuw.edu.pl, edmundp@mimuw.edu.pl}
\date{ }
\maketitle\markboth{ \bf A.Leroy,  J.Matczuk, E.Puczy{\l}owski}{ \bf
Quasi-duo $\mathbb{Z}$-graded rings}

\begin{abstract}A description of right (left) quasi-duo
 $\mathbb{Z}$-graded rings is given. It shows, in particular,  that a strongly
$\mathbb{Z}$-graded ring is left quasi-duo if and only if it is
right quasi-duo. This gives a partial answer to a problem posed by
Dugas and Lam in \cite{LD}.
\end{abstract}

A ring $R$ with an identity is called \cite{LD} $\it right$ ({\it
left}) {\it quasi-duo} if every maximal right (left) ideal of $R$
is two-sided. Quasi-duo rings were studied in many papers (Cf.
\cite{LD}, \cite{Y} and papers quoted there). The main open
problem in the area asks whether the classes of left and right
quasi-duo rings coincide (it is important,  as it concerns the
problem to what extend the notion of primitivity is left-right
symmetric, Cf. \cite{LD}). This problem was also an initial
motivation for our studies. Namely the results obtained in
\cite{LMP} on quasi-duo skew polynomial rings show that it would
be interesting to examine whether it could be possible to distinct
these classes within $\mathbb{Z}$-graded rings or, more generally,
to describe $\mathbb{Z}$-graded right (left) quasi-duo rings. The
methods of \cite{LMP} are rather specific for skew-polynomial
rings and one cannot apply them to $\mathbb{Z}$-graded rings. In
this paper we find another approach to that problem and describe
$\mathbb{Z}$-graded right (left) quasi-duo rings. This description
shows, in particular,  that a strongly $\mathbb{Z}$-graded ring
is right quasi-duo if and only if it is
  left quasi-duo. Thus, for strongly $\mathbb{Z}$-graded rings, the above
mentioned Dugas-Lam's problem has a positive solution.  As an
application we  also get back in another way the characterization
of right (left) skew polynomial and Laurent polynomial rings
obtained in \cite{LMP}.

The  results  on the Jacobson radical, the pseudoradical  and
maximal ideals of $\mathbb{Z}$-graded rings (see Proposition
\ref{anihilator}, Theorem \ref{pseudoradical}) can be of
independent interest.

 All rings in this paper are associative with
identity. To denote
 that $I$ is an ideal (left ideal, right ideal) of a ring $R$ we
will write $I\lhd R$ ($I<_lR$, $I<_rR$). The Jacobson radical of a
ring $R$ will be denoted by $J(R)$.

It is clear that $R$ is  right (left) quasi-duo if and only if
$R/J(R)$ is right (left) quasi-duo and that Jacobson semisimple
right (left) quasi-duo rings are subdirect sums of division rings,
so they are reduced rings. The class of right (left) quasi-duo
rings is closed under homomorphic images and finite subdirect sums
(Cf.\cite{LD}).

In what follows $\mathbb{Z}$ denotes the additive group of
integers and $R$ denotes a $\mathbb{Z}$-graded ring. Recall that
$R=\bigoplus_{n\in \mathbb{Z}} R_n$, the direct sum of additive
subgroups $R_n$, with $R_nR_m\subseteq R_{n+m}$ for all $n,m\in
\mathbb{Z}$. If $R_nR_m=R_{n+m}$, then $R$ is called {\it
strongly} {\it graded}.

Elements of $\bigcup_{n\in \mathbb{Z}}R_n$ are called {\it
homogeneous}.  Every $ r\in R$ can be written as a finite sum
$r=\sum_{m\leq i \leq n} r_i$, where $r_{i}\in R_{i}$ is called
the {\it homogeneous component} of $r$ of degree $i$. If $r_{m}$
and $r_{n}$ are nonzero, then  the length  $l(r)$ of $r$ is
defined as $n-m+1$. Clearly a nonzero element of $R$ is
homogeneous if and only if its length is equal to 1.

An ideal $I$ of $R$ is called {\it homogeneous} if
$I=\bigoplus_{n\in \mathbb{Z}} (I\cap R_n) $. The largest
homogeneous ideal contained in a given ideal $I$ of $R$ will be
denoted by $(I)_h$.

The following well known result of G.~Bergman (Cf. \cite{Rowen})
plays a substantial role in the paper.

\begin{theorem}\label{bergman}
For every $\mathbb{Z}$-graded ring  $R$
\begin{enumerate}
  \item[(i)] $J(R)$ is a homogeneous ideal;
  \item[(ii)]  If $r\in \bigcup_{0\neq n\in \mathbb{Z}}R_n$, then $1+r$ is
invertible if and only if $r$ is nilpotent.
   \end{enumerate}
\end{theorem}

A homogeneous ideal $P$ of $R$ is called {\it graded prime} if
$IJ\subseteq P$ implies $I\subseteq P$ or $J\subseteq P$ for
arbitrary homogeneous ideals $I$ and $J$ of $R$. It is well known
and not hard to check that if $P$ is a prime ideal of $R$, then
$(P)_h$ is a graded prime ideal of $R$. It is also  well known
that a homogeneous ideal of a $\mathbb{Z}$-graded ring is prime if
and only if it is graded prime.

The intersection of all nonzero
graded prime ideals of $R$ will be called the {\it graded
pseudoradical} of $R$. The empty intersection, by definition, is
equal to $R$.

The following result generalizes Lemma 3.2 from \cite{Matczuk2}.

\begin{theorem}\label{pseudoradical}
Suppose that a $\mathbb Z$-graded ring $R$ contains a maximal
ideal $M$ such that $(M)_h=0$. Then the graded pseudoradical of
$R$ is nonzero.
\end{theorem}

\begin{proof}

 Let  $a=\sum_{m\leq i\leq n}a_i$ be a nonzero element of
$M$ of minimal length, where $a_m\neq 0\neq a_n$. Since $(M)_h=0$,
$l(a)\geq 2$.

Let $C$ (resp. $D$) denote the sets of all $n$-th (resp. $m$-th )
components of nonzero elements from $M\cap (\bigoplus_{m\leq i\leq
n}R_i)$.  Notice that $C$ and $D$ are non empty homogeneous sets
depending only on $M$.

If $R$ has no nonzero graded prime ideals, then the graded
pseudoradical of $R$ is equal to $R$, so the thesis holds.

Suppose now that we can pick a nonzero graded prime ideal $Q$ of
$R$. Then $M+Q=R$, so $1=b+q$, where $b=\sum_{s\leq l\leq t}b_l\in
M$ and $1-b_0\in Q$ and $b_i\in Q$, for $s\leq i\leq t$, $i\neq
0$.  This implies that precisely one homogeneous component of
$b=\sum_{s\leq l\leq t}b_l\in M$ is not in $Q$.  Suppose that $b$
is an element in $M$ with the smallest possible length amongst the
elements of $M$ having precisely one homogeneous component not in
$Q$.  Let us write $b=\sum_{s\leq l\leq t}b_l\in M$ with
$b_k\notin Q$.

If $k\ne t$ we claim that $C\subseteq Q$.  If not, then there
exists $r=\sum_{m\leq i\leq n}r_i\in M$ such that $r_n\not\in Q$.
Since $Q$ is a prime graded ideal, there exists $c\in R_w$, for
some $w\in\mathbb{Z}$, such that $b_kcr_n\not\in Q$. Notice that
$n-m+1=l(r)\le l(b)= t-s+1$ and the element $u=bcr_n-b_tcr\in M$
is such that precisely one homogeneous component of $u$ (namely
$u_{k+w+n}$) is not in $Q$. Moreover, since
$(bcr_n)_l=(b_tcr)_l=0$ if $l<s+w+n$ and $u_{t+w+n}=0$, we get
$l(u)<l(b)$, which is impossible, by the choice of $b$. This
proves the claim.

If $k=t$ we can prove in a similar way that $D\subseteq Q$.

We conclude that $CRD\subseteq Q$ for any nonzero graded prime
ideal $Q$ of $R$.  Since $M$ is prime and $(M)_h=0$, the ring $R$
is a graded prime ring and hence $CRD\ne 0$.  This yields the
desired result.
\end{proof}

In what follows we denote by $\mathcal A$ the set of all maximal
right ideals $M$ of $R$ such that $R_n\not\subseteq M$, for some
$0\neq n\in \mathbb{Z}$ and by $\mathcal B$ the set of remaining
maximal right ideals of $R$. Set $A(R)=\bigcap_{M\in \mathcal A}M$
and $B(R)=\bigcap_{M\in \mathcal B}M$.

It is easy to describe $B(R)$. Note that $U=\sum_{0\neq n\in
\mathbb{Z}}R_{-n}R_n\lhd R_0$. It is clear that if $M\in \mathcal
B$, then $M=M_0+\bigoplus_{0\neq n\in \mathbb{Z}}R_n$ for a
maximal right ideal $M_0$ of $R_0$ containing $U$. Consequently
$B(R)=J+\sum_{0\neq n\in \mathbb{Z}}R_n$, where $J$ is the ideal
of $R_0$ containing $U$ such that $J(R_0/U)=J/U$.  In particular,
$B(R)$ is a two-sided ideal of $R$.

If $R$ is strongly graded,
then for every $0\neq n\in \mathbb{Z}$, $R_0=R_nR_{-n}$. This
shows that in this case ${\mathcal B}=\emptyset$, so $B(R)=R$ and
$A(R)=J(R)$.

Now we will describe $A(R)$. Let $A_l=\{ r\in R \mid R_nr\subseteq J(R)$, for every
$0\neq n\in \mathbb{Z} \}$ and $A_r=\{ r\in R \mid rR_n\subseteq
J(R)$, for every $0\neq n\in \mathbb{Z} \}$.

\begin{proposition}\label{anihilator}

 Let  $R$ be a $\mathbb{Z}$-graded ring. Then:
\begin{enumerate}
  \item[(i)] $A(R)=A_l=A_r$
  \item[(ii)] $A(R)\cap (\bigoplus_{0\neq n\in \mathbb{Z}}R_n)=J(R)\cap (\bigoplus_{0\neq n\in
 \mathbb{Z}}R_n)$.
 \end{enumerate}
\end{proposition}
\begin{proof}
  (i). It is clear that $A_l\lhd R$. Hence $A_lR_n<_lR$, for every $0\neq n\in
\mathbb{Z}$. Since $(A_lR_n)^2\subseteq J(R)$ and $R/J(R)$ is
semiprime, $A_lR_n\subseteq J(R)$. This proves that $A_l\subseteq
A_r$. Dual arguments give the opposite inclusion and show that
$A_l=A_r$.

Take any $M\in \mathcal A$. Then $R_n\not\subseteq M$, for some
$0\neq n\in \mathbb{Z}$. Obviously $(A_r+M)R_n\subseteq M$. Thus
$A_r+M\neq R$ and maximality of $M$ implies that $A_r\subseteq M$.
Consequently $A_r\subseteq A(R)$. Clearly $A(R)\cap B(R)=J(R)$,
$B(R)\lhd R$ and $A(R)<_rR$, so $A(R)B(R)\subseteq J(R)
$. Hence, since $\bigoplus_{0\neq n\in \mathbb{Z}}R_n\subseteq
B(R) $, we get that $A(R)\subseteq A_r$.

(ii). By (i), $A(R)R_n+R_mA(R)\subseteq J(R)$, for arbitrary $
n,m\in \mathbb{Z}\setminus \{0\}$. This implies that if $I$ is the
ideal of $R$ generated by $A(R)\cap (\bigoplus_{0\neq n\in
\mathbb{Z}}R_n)$, then $I^2\subseteq J(R)$. Consequently $A(R)\cap
(\bigoplus_{0\neq n\in \mathbb{Z}}R_n)\subseteq I\subseteq J(R)$.
Now it is easy to complete the proof of (ii).\end{proof}

\begin{theorem}\label{idealy maksymalne}
If a $\mathbb Z$-graded ring $R$ is right (left) quasi-duo, then
$R/M$ is a field, for every $M\in \mathcal A$.
\end{theorem}
\begin{proof}  We will prove the result when $R$ is right quasi-duo.
If $R$ is left quasi-duo, symmetric arguments can be applied. Let
$M\in \mathcal A$.
  Passing to the factor  ring  $R/(M)_h$, we can assume
without loss of generality that $(M)_h=0$. Since $R$ is right
quasi-duo, $R/M$ is a division ring. Making use of those two
facts, one can easily check that $R$ is a domain.  Moreover, by
Theorem \ref{pseudoradical}, the graded pseudoradical $P$ of $R$
is nonzero.

Let $0\ne n\in \mathbb{Z}$  and $a\in P_n=P\cap R_n $.  Clearly
$a$ is not nilpotent, as $R$ is a domain. Thus, by Theorem
\ref{bergman}, $1+a$ is not invertible. Hence there exists a
maximal right ideal $T$ of $R$ containing $1+a$. Since $R$ is
 quasi-duo, $T\lhd R$. Now $(T)_h$ is a prime homogeneous
ideal of $R$, so if $(T)_h\neq 0$, then $P\subseteq T$. This is
impossible as otherwise $1=(1+a)-a\in T$. Therefore $(T)_h=0$. Now
for every homogeneous element $b$ of $R$, $ab-ba=(1+a)b-b(1+a)\in
(T)_h=0$. This shows that $a$ belongs to the center $ Z(R)$ of $R$
and implies that  $P_n\subseteq Z(R)$, for all nonzero
$n\in\mathbb{Z}$.  Since $ M\in \mathcal A$, by definition,  there
exists $0\ne m\in \mathbb Z$ such that $R_m\not\subseteq M$. In
particular $R_m\ne 0$. Therefore, since $P$ is a nonzero
homogeneous ideal and $R$ is a domain, we can pick a nonzero
integer $n$ such that $P_n\ne 0$.  Then $ P_0P_n\subseteq
P_n\subseteq Z(R)$ and, as    $R$ is a domain, $P_0\subseteq Z(R)$
follows. The above implies that  $P\subseteq Z(R)$ and shows that
the division ring $R/M=(M+P)/M$ is   commutative, i.e. it is a
field.
\end{proof}

\begin{theorem}\label{charakteryzacja}
 A $\mathbb{Z}$-graded ring
$R$ is  right (left) quasi-duo if and only if $R_0$ is right
(left) quasi-duo and $R/A(R)$ is a commutative ring.
\end{theorem}

\begin{proof}
 Suppose that $R$ is  right quasi-duo. Let $M$ be a maximal right ideal
of $R_0$. Clearly $MR$ is a proper right ideal of $R$.
Consequently $MR$ is contained in a maximal right ideal $T$ of
$R$. Since $R$ is right quasi-duo, $T\lhd R$. It is clear that
$M=T\cap R_0$, so $M\lhd R_0$. Thus $R_0$ is a right quasi-duo
ring.

When $\cal A\ne \emptyset$,   Theorem \ref{idealy maksymalne}
implies that $R/A(R)$ is a subdirect sum of fields, so it is a
commutative ring. If $\cal A= \emptyset$, then $A(R)=R$ and the ring $R/A(R)$ is
also commutative.

Suppose now that $R_0$ is  right quasi-duo and $R/A(R)$ is
commutative. Let $I$ be the ideal of $R$ generated by
$\bigcup_{0\neq n\in \mathbb{Z}}R_n$. Then, by Proposition
\ref{anihilator}(i), $IA(R)\subseteq J(R)$. Hence $(I\cap
A(R))^2\subseteq J(R)$ and semiprimeness of $J(R)$ implies that
$I\cap A(R)\subseteq J(R)$. This shows that $R/J(R)$ is a
homomorphic image of a subdirect sum of rings $R/I$ and $R/A(R)$.
Clearly $R/I$ is a homomorphic image of $R_0$. Consequently both
$R/I$ and $R/A(R)$ are right quasi-duo, so, further, $R/J(R)$ and
$R$ are right quasi-duo.

When $R$ is left quasi-duo, symmetric arguments apply.
\end{proof}
  Theorem \ref{charakteryzacja} immediately gives the following

\begin{corollary}
\label{cor2}
 Suppose a $\mathbb{Z}$-graded
ring $R$ is right quasi-duo. Then: \\
1.  $R_0$ is right  quasi-duo;\\
2. $R$ is left quasi-duo iff $R_0$ is left quasi-duo.
\end{corollary}
We know, by the remark made just before Proposition
\ref{anihilator},  that $A(R)=J(R)$, provided   $R$ is strongly
$\mathbb{Z}$-graded. Thus, by Theorem \ref{charakteryzacja}, we
get:

 \begin{corollary}
 \label{cor1}
 Suppose that $R$ is strongly $\mathbb
Z$-graded. Then $R$ is right quasi-duo iff $R$ is left quasi-duo
iff $R/J(R)$ is commutative.
\end{corollary}

 Now, as an application of Theorem \ref{charakteryzacja}, we will
get   characterizations of right (left) quasi-duo skew polynomial
rings and skew Laurent  polynomial rings obtained in \cite{LMP}.

Let $\sigma$ be an endomorphism of a ring $S$ and $S[x;\sigma]$ be
the associated skew polynomial ring  with coefficients from $S$
written on the left. Denote by $N(S)$ the set $\{ s\in S\mid
s\sigma (s)\cdots {\sigma}^n(s)=0$, for some positive integer
$n\}$. Clearly $N(S)= \{ s\in S\subseteq S[x;\si] \mid (sx)^n=0$,
for some positive integer $n\}$.  Let $N(S)[x;\si]$ be the set of
all polynomials from $S[x;\si]$ which have all their coefficients
in $N(S)$. Notice also that $\si(N(S))\subseteq N(S)$. Thus, if
$N(S)\lhd S$ then $N(S)[x;\si]\lhd S[x;\si]$, $\si$ induces an
endomorphism, also denoted by $\si$, on $S/N(S)$ and
$(S/N(S))[x;\si]\simeq S[x;\si]/N(S)[x;\si]$.

\begin{lemma}
\label{radykal Jacobsona poly} Suppose that    the skew polynomial
ring $S[x;\si]$  is right (left) quasi-duo. Then
$J(S[x;\si])\subseteq N(S)[x;\si]\subseteq A(S[x;\si])$.
 \end{lemma}

\begin{proof}   Since $S[x;\si]$  is right (left) quasi-duo, the ring
$S[x;\si]/J(S[x;\si])$  is reduced,  so every nilpotent element of
$S[x;\si]$ belongs to $J(S[x;\si])$. Thus, in particular,
$xN(S)\subseteq J(S[x;\si])$ and consequently $Sx^nN(S)\subseteq
J(S[x;\si])$, for all  $n> 0$. The ring $S[x;\si]$ is
$\mathbb{Z}$-graded in the canonical way and the last inclusion
together with Proposition \ref{anihilator}(i) yield $N(S)\subseteq
A(S[x;\si])$. This shows that $N(S)[x;\si] \subseteq A(S[x;\si])$.

Let $ax^n\in J(S[x;\si])$, for some $n>0$. Then, by Theorem
\ref{bergman}, $ax^n$ and $x^na$ are also  nilpotent elements of
$S[x;\si]$ and so $x^na\in J(S[x;\si])$. Hence
$Sx^mx^{n-1}a\subseteq J(S[x;\si])$, for all $m>0$ and Proposition
\ref{anihilator}(i) shows that $x^{n-1}a\in J(S[x;\si])$.
Repeating this procedure we obtain $xa\in J(S[x;\si])$ and Theorem
\ref{bergman} implies that $a\in N(S)$. Since $J(S[x;\si])$ is a
homogenous ideal, we obtain $J(S[x;\si])\subseteq N(S)[x;\si]$.
\end{proof}

\begin{corollary}
{\em (\cite{LMP})} $S[x;\sigma]$ is right (left) quasi-duo if and
only if $S$ is right (left) quasi-duo, $N(S)\lhd S$,
$J(S[x;\sigma])=J(S)\cap N(S)+N(S)[x;\sigma]x$ and
$(S/N(S))[x;\sigma]$ is a commutative ring.
\end{corollary}
\begin{proof}
 Suppose that the ring $S[x;\sigma]$ is right (left) quasi-duo.
  Then, by  Proposition \ref{anihilator}(i),
$A(S[x;\sigma]x)\subseteq J(S[x;\sigma])$. Thus, by  Lemma
\ref{radykal Jacobsona poly}, we get $A(S[x;\si])=N(S)[x;\si]$.
This implies that $N(S)$ is an ideal of $S$. Now, by Theorem
\ref{charakteryzacja}, the ring $(S/N(S))[x;\si]\simeq
S[x;\si]/N(S)[x;\si]$ is commutative.

 Since
$B(S[x;\sigma])=J(S)+S[x;\sigma]x$ and $J(T)=A(T)\cap B(T)$, we
also obtain $J(S[x;\sigma])=J(S)\cap N(S)+N(S)[x;\sigma]x$.

Conversely, by making use of  Proposition \ref{anihilator}(i),  it
is evident that when  $J(S[x;\sigma])=J(S)\cap
N(S)+N(S)[x;\sigma]$, then $A(S[x;\sigma])=N(S)[x;\sigma]$. Now if
the ring $(S/N(S))[x;\sigma]$ is  commutative and $S$ is right
(left) quasi-duo, then $S[x;\sigma]$ is right (left) quasi-duo, by
Theorem \ref{charakteryzacja}.
\end{proof}

\begin{corollary}{\em (\cite{LMP})}
Let $\sigma$ be an automorphism of a ring $S$. Then the skew
Laurent
 polynomial ring $S[x,x^{-1};\sigma]$ is right (left) quasi-duo if and only
if $N(S)\lhd S$, $J(S[x,x^{-1};\sigma])=N(S)[x,x^{-1};\sigma]$ and
$(S/N(S))[x,x^{-1};\sigma]$ is a commutative ring.
\end{corollary}

\begin{proof}
 Since $S[x,x^{-1};\sigma]$ is a strongly graded,
$A(S[x,x^{-1};\sigma])=J(S[x,x^{-1};\sigma])$.

 Suppose now that $S[x,x^{-1};\sigma]$ is right (left) quasi-duo.
 Then, as    $N(S)x$ consists of nilpotent elements, $N(S)[x,x^{-1};\sigma]\subseteq
J(S[x,x^{-1};\sigma])$. The opposite inclusion follows immediately
from Theorem \ref{bergman}. Obviously $N(S)\lhd S$ and by Theorem
\ref{charakteryzacja}, $(S/N(S))[x,x^{-1};\sigma]$ is a
commutative ring. This proves the only if" part. The "if" part is
a direct consequence of Theorem \ref{charakteryzacja}. \end{proof}


\begin{thebibliography}{99}
\bibitem{LD}
T.Y. Lam, A.S. Dugas, Quasi-duo rings and stable range descent, J.
Pure and Appl. Alg. 195 (2005) 243--259.

\bibitem{LMP} A. Leroy, J. Matczuk and E.R. Puczy\l owski,
Quasi-duo skew polynomial rings, J. Pure and Appl. Alg. 212, 2008,
1951--1959.



\bibitem{Matczuk2} J. Matczuk, Maximal ideals of skew polynomial
rings of automorphism type,   Comm. Algebra 24(3), 1996, 907--917.

\bibitem{Rowen} L.H.  Rowen, \it Ring Theory, vol. I, \rm
Pure and Applied Mathematics vol. 127, Academic Press, 1988.

 \bibitem{Y}  H.~P.~Yu, On quasi-duo
rings, Glasgow Math. J. 37,  1995, 21--31.
 \end{thebibliography}
\end{document}